\documentclass{article}

\usepackage{enumerate}
\usepackage{amsmath,amsthm}
\usepackage{amssymb, latexsym}
\usepackage[mathscr]{euscript}
\usepackage{color}
\usepackage{array}
\usepackage{flafter}
\usepackage{layout}
\usepackage{graphicx}
\usepackage{epstopdf}
\usepackage[ansinew]{inputenc}
\usepackage{xfrac}
\usepackage{algorithm}
\theoremstyle{plain}
  \newtheorem{teo}{Theorem}
  
  \newtheorem{lema}{Lemma}
  \newtheorem{prop}{Proposition}
  
\theoremstyle{definition}

\theoremstyle{remark}
  \newtheorem{remark}{Remark}

\title{On the explanatory power of principal components}

\author{Daniel A. D\'{\i}az--Pach\'on$^1$\thanks{Ddiaz3@med.miami.edu}, J. Sunil Rao$^1$\thanks{jrao@biostat.med.miami.edu} and Jean-Eudes Dazard$^2$\thanks{jxd101@case.edu}\\
$^1$ University of Miami, Miami, FL, USA\\
$^2$ Case Western Reserve University, Cleveland, OH, USA\\
}
\date{\today}

\begin{document}
\maketitle

\begin{abstract}
We show that if we have an orthogonal base ($u_1,\ldots,u_p$) in a $p$-dimensional vector space, and select $p+1$ vectors $v_1,\ldots, v_p$ and $w$ such that the vectors traverse the origin, then the probability of $w$ being to closer to all the vectors in the base than to $v_1,\ldots, v_p$ is at least 1/2 and converges as $p$ increases to infinity to a normal distribution on the interval [-1,1]; i.e., $\Phi(1)-\Phi(-1)\approx0.6826$. This result has relevant consequences for Principal Components Analysis in the context of regression and other learning settings, if we take the orthogonal base as the direction of the principal components.
\end{abstract}

\section{Introduction}\label{Intro}

Principal Components Analysis (PCA) has been popular since its inception due to its two main features: 1. Explaining well the variation of the information and 2. Because of the first one, being able to reduce the space to those variables that explain most of the variation. In the context of regression, PCA has been used to explain the response in terms of the Principal Components (PC) of the original input variables (see for instance, \cite{Kendall1957},\ p. 75; \cite{Hocking1976}; \cite{MostellerTuckey1977},\ p. 307; \cite{Scott1992}). However, because one of the main features of PCA is dimension reduction (feature 2 mentioned above) it is not always clear that PCA can be used and to still get that the response is well explained by the reduced space of the main PC (see for instance \cite{Cox1968},\ p. 72; \cite{Joliffe1982}; \cite{HadiLing1998}). In fact, as early as 1957, in one of the first references to PCA in regression, Hotelling was doubtful of the method because of this same reason (\cite{Hotelling1957}). For a historical review of PCA in regression settings see \cite{Cook2007}.

Note the interesting fact that the earliest of our references advocating the use of PCA in regression (\cite{Kendall1957}) was proposed in the same year that its first criticism was known (\cite{Hotelling1957}). The questionings of PCA in regression seem to be as old as the method itself. Now, the method and its questionings transcend regression going further to more general learning settings. But whether regression or any other learning setting, the two interesting questions to ask in order to determine the usefulness of the method are the following: 1. Is it the case that (most of the times) the response variable is better explained by the PC than by the original predictors? And, if the answer to this question is yes, then 2. Are the leading PC more able to explain the response than the rest of the PC?

Question 2, as mentioned above, was posed from the beginning in order to determine the usefulness of PCA in learning settings. And this question has been answered recently by Artemiou and Li (\cite{ArtemiouLi2009}). They proved that if there are $p$ PC, and $1\leq i\leq j\leq p$, then the probability of the response being closer to the $i$-th PC than to the $j$-th PC is higher than the probability of the complementary event ---the response being closer to the $j$-th PC than to the $i$-th PC. However, Question 2 only is important if Question 1 has positive answer, because if the original input variables explain better the answer than the PC, then a PCA is not justified. Therefore, we focus our interest here in Question 1. The aim here is to propose an explanation for the answer to this question showing, in general, that the probability of all the PC explaining better (than all the original variables) a response variable is high even in low dimensions, and increases exponentially as $p$ grows to infinite.

The main result of this article is Theorem \ref{TeoPCs}. It answers partially Question 1 on why PC usually do a better job explaining a response than the original input variables. Theorem \ref{TeoPCs} is in fact a particular version of Theorem \ref{TeoLines}, a geometric result that establishes the following: let us assume we have a given orthogonal base of a $p$-dimensional vectorial space, assume also that we select according to a uniform distribution a $p$-dimensional base for the same space (so it does not have two orthogonal vectors a.s., and it generates the space a.s.). Then, if we select a new vector uniformly at random, this new vector is going to be closer (in terms of angular distance) to the orthogonal base than to the random base with probability no less than 1/2 for every $p$ and this probability converges approximately to 0.6826 as $p\rightarrow\infty$.

\section{Setting}

Consider the space $\mathbb R^p$ with  orthogonal base $(u_1,\ldots,u_p)$. Take $\Omega$ to be the space of all unidimensional lines traversing the origin. Select  from $\Omega$ uniformly and independently a set of $p+1$ lines $v_1,\ldots,v_p$ and $w$. Note that $v_1,\ldots,v_p$ form almost surely a basis for $\mathbb R^p$. Here \textit{uniformly} means that we have a unit closed ball centered around the origin $B[0,1]$ with surface $\tilde B$. Then, calling $\mathcal L_p$ the Lebesgue measure in $p$ dimensions, for every $A\subset\tilde B$, we have that the probability of an element in $\Omega$ to traverse $A$ is given by

\begin{align}\label{Unif}
	\frac{\mathcal L_{p-1}(A\cup-A)}{\mathcal L_{p-1}\tilde B},
\end{align}
where $-A=\{x\in\tilde B:-x\in A\}$.

Let $\tilde\phi_i$ be the angle between $w$ and $u_i$, and define $\phi_i:=\min\left\{\left|\tilde\phi_i\right|,\left|\pi-\tilde\phi_i\right|\right\}$.  Let $\tilde\varphi_i$ be the angle between $w$ and $v_i$, and define $\varphi_i := \min\left\{\left|\tilde\varphi_i\right|,\left|\pi-\tilde\varphi_i\right|\right\}$. Let $\tilde\theta_{ij}$ be the angle between $w_j$ and $v_i$, and define $\theta_{ij} := \min\left\{\left|\tilde\theta_{ij}\right|,\left|\pi-\tilde\theta_{ij}\right|\right\}$.

Call $\mathcal O_w$ the orthant (hyper-octant) where $w$ is located in, and, analogously, $\mathcal O_{u_i}$ and $\mathcal O_{v_j}$ the orthants where $u_i$ and $v_j$ live in, respectively. For convenience, let us refer to the orthants by the combination of signs they correspond to (i.e., when $p=2$ the four quadrants can be denoted by $(+,+),(+,-),$ $(-,+),(-,-)$).

Figure \ref{fig1} gives a geometric representation of expectations of uniformly and independently chosen random variables ($v_1$, $v_2$) for $p=2$. Figure \ref{fig2} gives a geometric representation of realizations of uniformly and independently chosen random variables ($v_1$, $v_2$) and $w$ for $p=2$ illustrating how which set, ($v_1$, $v_2$) or ($u_1$, $u_2$), explains better the r.v.\ $w$ in terms of angular distance.
\begin{figure}[!hbt]
    \centering
    \includegraphics[width=1.0\textwidth]{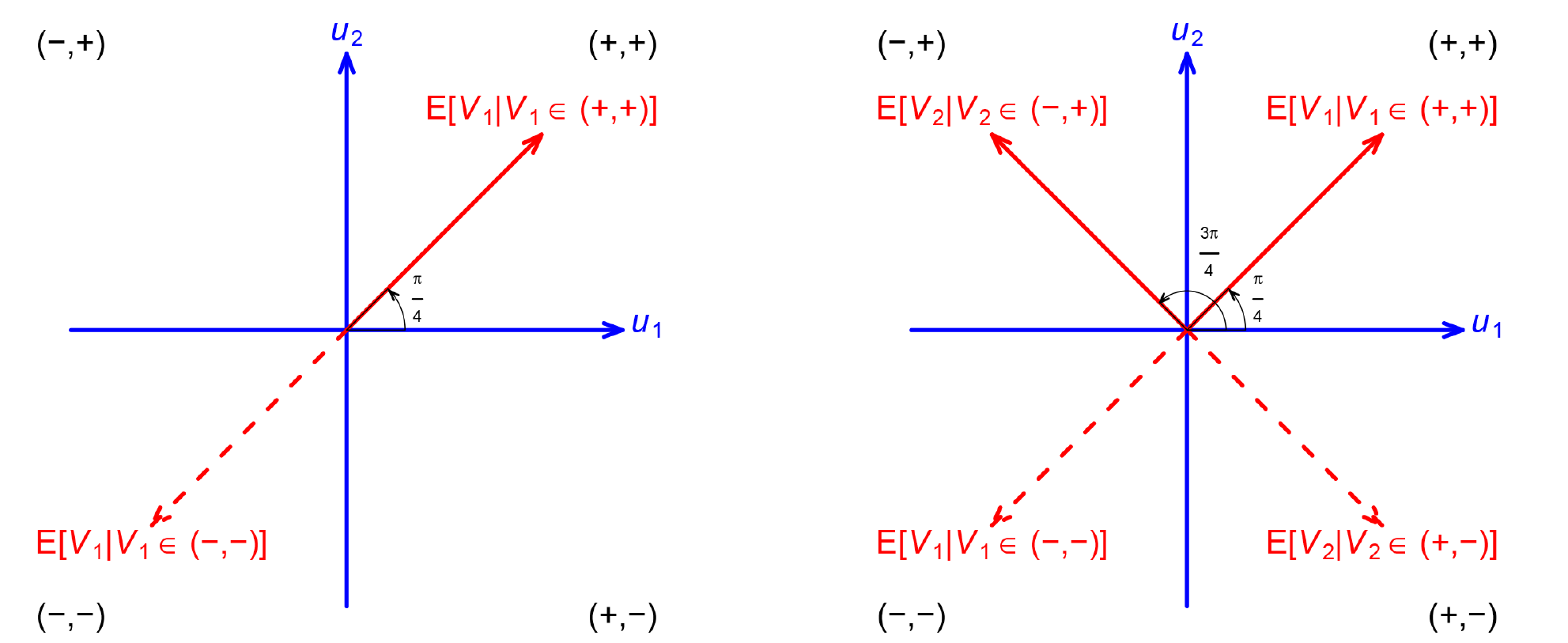}
    \caption{Geometric representation for $p=2$ of the conditional expectation of $v_1$ and $v_2$ (red vectors) given the quadrant in which they are living. Blue vectors represent $u_1$ and $u_2$.}\label{fig1}
\end{figure}

\begin{figure}[!hbt]
    \centering
    \includegraphics[width=1.0\textwidth]{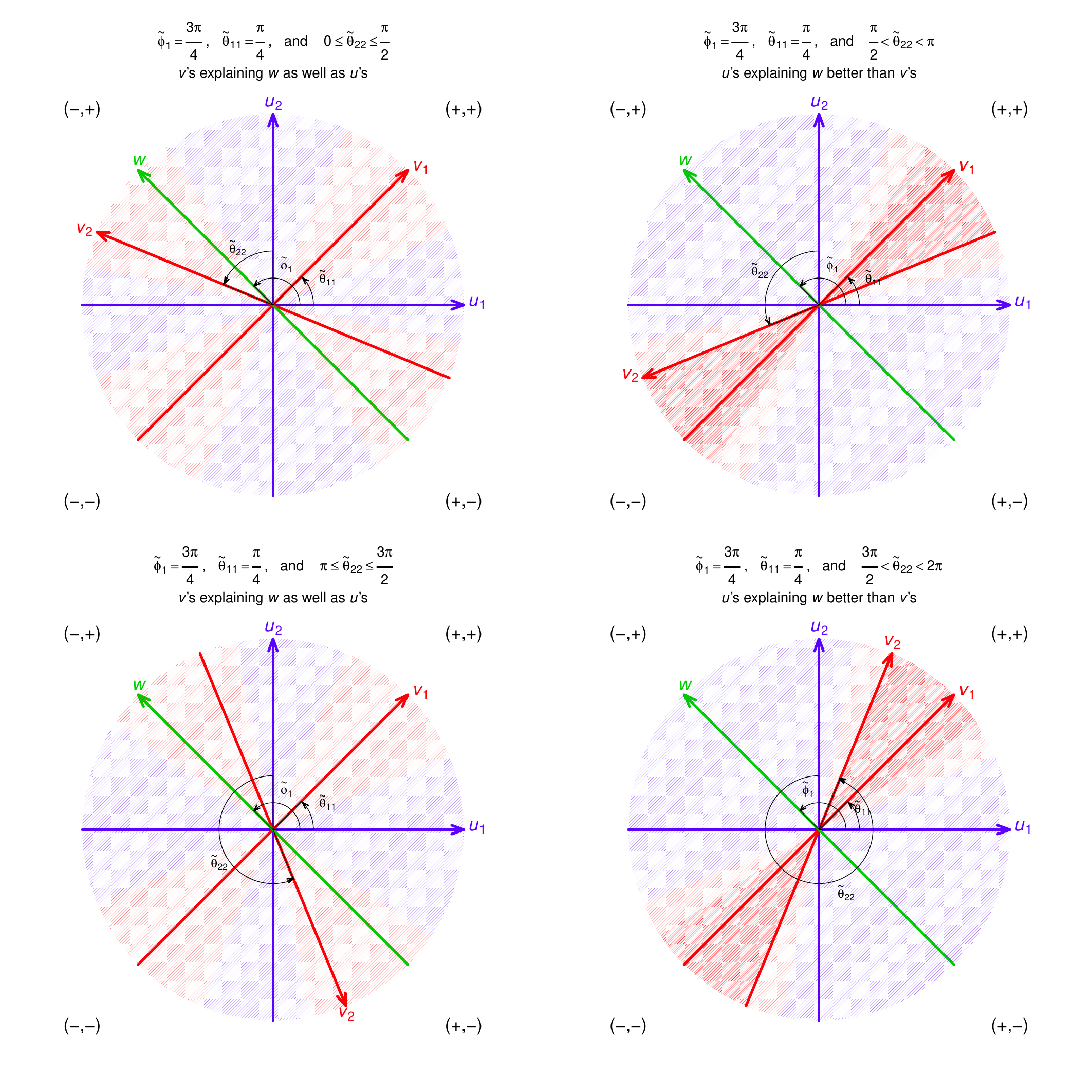}
    \vskip -10pt
    \caption{Geometric representation for $p=2$ of $v_1,v_2$ and $w$. Blue vectors represent $u_1$ and $u_2$, red vectors represent realizations of $v_1$ and $v_2$, and the green vector represents $w$. Areas of circle sectors of $u$ explaining $w$ better than $v$ are shaded in blue, and sectors of $v$ explaining $w$ better than $u$ are shaded in red. The overlapping sectors of variables $v$ appear in darker red. Notice that when ``$u$ explains $w$ better than $v$'', then the red shaded regions of $v_1$ and $v_2$ overlap. Also, notice that the case of ``$u$ explaining $w$ as well as $v$'' includes the situation where ($v_1$, $v_2$) are mutually orthogonal. In all plots, illustrations are for an arbitrary realization of $w : \tilde\phi_{1} = \widehat{(u_1,w)} = \frac{3\pi}{4}$, as well as $v_1 : \tilde\theta_{11} = \widehat{(u_1,v_1)} = \frac{\pi}{4}$, and any realizations of $v_2 : \tilde\theta_{22} = \widehat{(u_2,v_2)}$ in the indicated intervals.}\label{fig2}
\end{figure}
\vskip 0pt

We prove the following result:
\begin{teo}\label{TeoLines}
    \begin{enumerate}
        \item For $k$ fixed and $i\in\{1,\ldots,p\}$, the probability of $w$ being closer to $u_k$ than to $v_i$ is bigger than or equal 1/2.
        \item For $k$ fixed and $i\in\{1,\ldots,p\}$, the probability of $w$ being closer to $u_k$ than to $v_i$ converges to $\Phi(1)-\Phi(-1)\approx 0.6826$.
    \end{enumerate}
\end{teo}

\section{Proof of Theorem \ref{TeoLines}}

Call $E_w$ the event ``$w$ is located in the orthant $\mathcal O_w$'' and, analogously, let $E_{v_i}$ as the event ``$v_i$ is living in the orthant $\mathcal O_{v_i}$''. Because of (\ref{Unif}) these events have each probability $2^{1-p}$. Assume without loss of generality that $\mathcal O_w$ is $(+,\,\cdots,+)$. Then
\begin{align*}
	\textbf E[w|E_w]=(1,\ldots,1)p^{-1/2},
\end{align*}
and from this we obtain
\begin{align}\label{cosw}
	\textbf E[\cos\phi_k|E_w]=p^{-1/2},
\end{align}
for $k=1,\ldots,p$. Analogously to (\ref{cosw}), we have that
\begin{align*}
	\textbf E[\cos\theta_{ik}|E_{v_i}]=p^{-1/2},
\end{align*}
for $i,k=1,\ldots,p$. Now, for $0<j\leq p/2$, let $j$ be the number of different signs between $\mathcal O_w$ and $\mathcal O_{v_i}$. Then we obtain that
\begin{align}\label{cosZW}
	\textbf E\left[\cos\varphi_i|E_w,E_{v_i}\right]=1-2jp^{-1},
\end{align}
and, from (\ref{cosw}) and (\ref{cosZW}),
\begin{align}\label{cosdifs}
	\textbf E\left[\cos\phi_k-\cos\varphi_i|E_w,E_{v_i}\right]=p^{-1/2}+2jp^{-1}-1,
\end{align}
so we get that (\ref{cosdifs}) is  nonnegative  iff $\left\lceil\frac{p-p^{1/2}}{2}\right\rceil\leq j\leq\lfloor p/2\rfloor$. Call $p^*:=\left\lceil\frac{p-p^{1/2}}{2}\right\rceil$. Then, for $G_{ik}^+:=\{\cos\phi_k-\cos\varphi_i\geq0\}$ we get that,
\begin{align}
	\textbf 1\left\{G_{ik}^+|E_w,E_{v_i}\right\}=
	\begin{cases}
		1 \text{\ \ \ \ \ \ if } p^*\leq j\leq\lfloor d/2\rfloor \\
		0 \text{\ \ \ \ \ \ if } 0\leq j\leq p^*-1,
	\end{cases}
\end{align}
therefore
\begin{align}\label{dif11}
	\textbf E\textbf 1\left\{G_{ik}^+|E_w, E_{v_i}\right\}=
	\begin{cases}
		2^{1-p}\sum_{j=p^*}^{\lfloor p/2\rfloor}\binom{p}{j} & \text{ if $p$ is odd}, \\
		2^{1-p}\left(\sum_{j=p^*}^{p/2-1}\binom{p}{j}+\frac{\binom{p}{p/2}}{2}\right) & \text{ if $p$ is even}.
	\end{cases}
\end{align}

As a visualization of the rate of convergence, we show in Figure \ref{fig3} the corresponding convergence plot of $\textbf E\textbf 1\left\{G_{ik}^+|E_w, E_{v_i}\right\}$ as a function of $p$.

\begin{figure}[!hbt]
    \centering
    \includegraphics[width=1.0\textwidth]{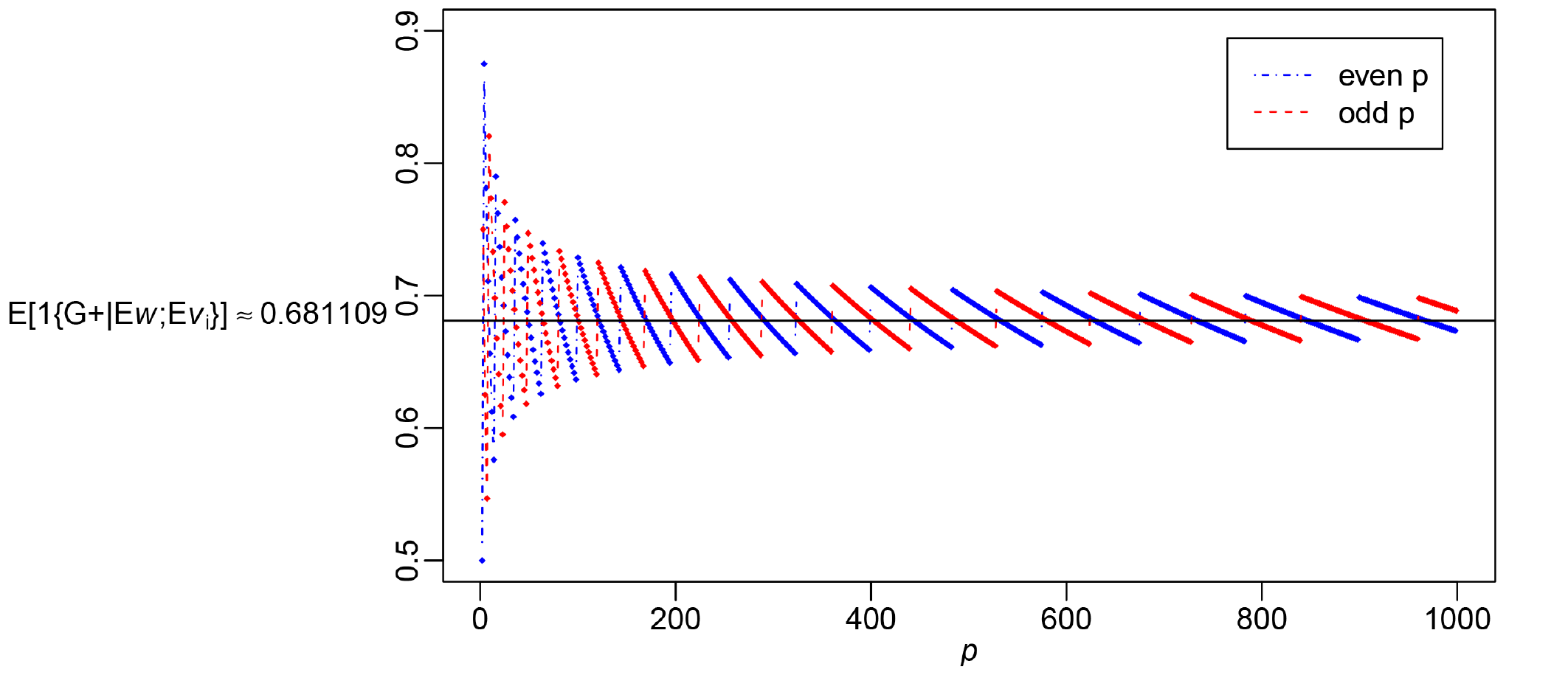}
    \vskip -10pt
    \caption{Convergence plot of $\textbf E\textbf 1\left\{G_{ik}^+|E_w, E_{v_i}\right\}$ for $p\in \{2,\dots,993\}$.}\label{fig3}
\end{figure}

Now, by (\ref{Unif}), total probability and independence,
\begin{align}\label{ProbG}
	\textbf P\left[G_{ik}^+\right]=\textbf E\textbf 1\left\{G_{ik}^+|E_w, E_{v_i}\right\}.
\end{align}

\begin{prop}\label{CondExp}
	$\textbf P\left[G_{ik}^+\right]\geq1/2$, for all $p\geq2$.
\end{prop}

Some comments are in order before proceeding. Because of uniformity we are only considering half of the space ($2^{p-1}$ orthants), since for all $\omega\in\Omega$, if $\omega$ lives in a given orthant $\mathcal O_\omega$, then it is also living in $-\mathcal O_\omega$. From this observation we see why all the combinations considered in (\ref{dif11}) are located on the left half of the Pascal triangle for each $p$. Call $T(p)$ the total number of terms in this half of the Pascal triangle; then we have $T(p)=\lfloor p/2\rfloor$. For a given $p$, call $R(p)$ the number of terms being added in (\ref{dif11}), and we have that $R(p)=\lfloor p/2\rfloor-p^*+1$. Call also $L(p)=T(p)-R(p)$. Finally, let $D(p)$ be $\textbf P\left[G_{ik}^+\right]-\textbf P\left[\left(G_{ik}^+\right)^c\right]$, then the statement in Proposition \ref{CondExp} is equivalent to prove that $D(p)\geq0$ and it is enough to look at the numerators of these probabilities.

\begin{lema}\label{SimpleProp} $L(p)$, $C(p)$, $T(p)$ and $D(p)$ satisfy the following properties:
    \begin{enumerate}
        \item[(1)] $T(p)=p/2$ for all even $p$.
        \item[(2)] $T(p+1)=T(p)$ for all even $p$.
        \item[(3)] $T(p+2)=T(p)+1$ for all $p$.
        \item[(4)] $L$ is increasing and either $L(p+1)\in\{L(p),L(p)+1\}$.
        \item[(5)] $R(p+1)\in\{R(p)-1,R(p),R(p)+1\}$
        \item[(6)] $R(p+2)\in\{R(p),R(p)+1\}$.
        \item[(7)] $L(p+2)\in\{L(p),L(p)+1\}$.
        \item[(8)] Let $p\geq2$ and define $\mathcal L_l:=\{p\in\mathbb N:L(p)=l\}$. Then $\#\mathcal L_l$, the cardinality of $\mathcal L_l$, is at least 2 and at most 3 for all $p\geq2$.
        \item[(9)] For $p,p'\in\mathcal L_l$ and $p<p'$, we have that $2^pD(p)<2^{p'}D(p')$.
    \end{enumerate}
\end{lema}

\begin{proof}
(1), (2) and (3) follow from the definitions. (4) follows from the fact that $\{p^*\}$ is an increasing sequence and $(p+1)^*-p^*\leq1$. The following ones can be easily deducted from the former properties.
\end{proof}

Note that $D(p)$ is not monotonically increasing. When $L(p+1)=L(p)+1$, we have that $R(p+1)$ is either $R(p)$ or $R(p)-1$. For such cases, $D(p+1)<D(p)$. However, the next proposition proves that despise this fact, the subsequence $\{\min_{p\in\mathcal L_l}D(p)\}_{l=0}^\infty$ is increasing.

\begin{lema}\label{IncreasingSubseq}
$S(\min\mathcal L_l)>S(\min\mathcal L_l-2)$, for all $l>0$.
\end{lema}

\begin{proof}
First, note that for all $l\geq1$, we have that $L(\min\mathcal L_l)=L(\min\mathcal L_l-2)+1$. Thus, $R(\min\mathcal L_l)=R(\min\mathcal L_l-2)$ (see Properties (3), (6) and (7) in Lemma \ref{SimpleProp}). Using this fact and basic combinatorics properties, it is easily seen that $D(p+2)=4D(p)-2\left(\binom{p}{p^*}-\binom{p}{p^*-1}\right)$, which proves the Lemma.
\end{proof}

\begin{proof}[Proof of Proposition \ref{CondExp}]
Lemmas \ref{SimpleProp} and \ref{IncreasingSubseq}, together with the fact that $D(2)=D(\min\mathcal L_0)=1$, prove the result.
\end{proof}
\begin{proof}[Proof of Theorem \ref{TeoLines}]
	Part 1 follows directly from Proposition \ref{CondExp}.
	
	To prove the Part 2, let $Y_p$ be a Bin$(p,1/2)$. By the central limit theorem we have:
	\begin{align*}
		\textbf P\left[Y_p<(p-p^{1/2})/2\right]\rightarrow\Phi(-1),
	\end{align*}
	therefore, by symmetry of the sequence $\binom{p}{0},\binom{p}{1},\ldots,\binom{p}{p}$, we obtain the result.
\end{proof}

As a visualization of the rate of convergence, we show in Figure \ref{fig3} the corresponding convergence plot of (\ref{dif11}) as a function of $p$.

\begin{remark}\label{R8}
Seeing (\ref{ProbG}) from the viewpoint of the unconditioned probability of $G^+_{ik}$ at the left-hand side, we are looking at the event of, given $i$ and $k$, having $u_k$ closer to $w$ than $v_i$. Therefore, given $k$, the r.v.\ $U_k\sim$Bin$(p,\textbf P[G^+_{ik}])$ tells us the probability that $u_k$ explains better $w$ than $U_k$ of the $v$'s. For instance, since $\textbf P[G^+_{ik}]\approx 2\textbf P[(G^+_{ik})^c]$ (indeed, it can be shown that $\textbf P[G^+_{ik}]>2\textbf P[(G^+_{ik})^c]$ for $p\geq 783$), then $\textbf P[U_k=p]\approx 2^p\textbf P[U_k=0]$. That is, the probability of $u_k$ explaining better $w$ than all the $u$'s is exponentially higher than the probability of all the $v$'s explaining better $w$ than $u_k$.
\end{remark}

\begin{remark}\label{R9}
On the other hand, the right-hand side of (\ref{ProbG}), because of the condition on $E_w$, implies that in expectation $w$ is closer to \emph{all} $u_1,\ldots,u_p$ than to $v_i$, since being in ``the middle'' of an orthant means that $w$ is at the same distance of each $u_i$, $i=1,\ldots,p$. Therefore, as in the previous Remark, for large $p$, Theorem \ref{TeoLines} reveals the conditionally expected proportion of $v$'s that do not explain $w$ as well as \emph{all} the $u$'s. The value is around 2/3. The downside of this conditioning is that around 1/3 of the $v$'s are expected to explain better $w$ than all the $u$'s.
\end{remark}

\begin{remark}
Finally, note that having compared the distance from $w$ to $v_i$ and to $u_k$ as we did above, it is difficult to add to the analysis a second orthogonal vector $u_{k'}$, with $k\neq k'$, since the angle between $w$ and $u_{k'}$ is going to be dependent on the angle betwen $w$ and $u_k$.
\end{remark}

\section{Extension to random variables}

One advantage of the approach used in the former section is the relatedness of the cosine similarity and the correlation coefficient of two random variables. The only difference being that the correlation coefficient is invariant to translations, while the cosine similarity is not. From the geometrical point of view, the problem is solved by restricting ourselves to consider only the vectors traversing the origin of the $p$-dimensional Euclidean space, precisely in the way $\Omega$ is defined. From the probabilistic point of view, however, this restriction does not affect, since there is no loss of generality once we consider random vectors centered around the origin of the space:

Let $\Omega'$ be the space of all the continuous random variables in $\mathbb R^p$ with expected value 0 and such that any $p$-vector $\textbf X$ of $p$ of them has a covariance matrix $\Sigma$. Let $\textbf Y$ be the set of principal components of $\textbf X$. Then the $p$ lines partition the space in (almost) equal regions. Consider any continuous random variable $Z$ living in $\mathbb R^p$ such that $\textbf EZ=0$. Let $\Omega$ be the space of unidimensional supports of the random variables in $\mathbb R^p$.  We have the following result:

\begin{teo}\label{TeoPCs}
    \begin{enumerate}
        \item For $k$ fixed and $i\in\{1,\ldots,p\}$, we have that
        \begin{align*}
            \textbf P\left[\left|\rho_{Y_kZ}\right|\geq\left|\rho_{X_iZ}\right|\right]\geq1/2.
        \end{align*}
        \item For $k$ fixed and $i\in\{1,\ldots,p\}$, we have that
        \begin{align*}
            \textbf P\left[\left|\rho_{Y_kZ}\right|\geq\left|\rho_{X_iZ}\right|\right]\rightarrow\Phi(1)-\Phi(-1)\approx 0.6826.
        \end{align*}
    \end{enumerate}
\end{teo}

\begin{proof}
Take $S_{Y_1},\ldots,S_{Y_p}$ to be the (unidimensional) supports of $Y_1.\ldots,Y_p$, respectively. Take $S_{X_1},\ldots,S_{X_p}$ the (unidimensional) supports of $X_1,\ldots,X_p$, respectively. Take $S_Z$ to be the support of $Z$. Then the result follows from Theorem \ref{TeoLines}.
\end{proof}

Some comments are in order here:
\begin{itemize}
 \item In the context of geometry, Theorem \ref{TeoLines} is true for every set of orthogonal vectors, but their being orthogonal through uniform random selection is not happening a.s. However, even if $v_1, \ldots, v_p$ are arbitrarily selected to be orthogonal, then it can be seen that they explain the space as good as $u_1, \ldots, u_p$, but no better.
 \item Because of the former point, the uncorrelated random variables do not need to be the set of principal components. However, the condition of knowing the covariance matrix $\Sigma$ allows us to identify a particular set of mutually orthogonal random variables for which we know explicitly the joint distribution and the marginals.
 \item When it comes to PCA we are not doing two things: First, we are not using the variance of the PC in any sense in this proof, except to find the directions of the components themselves. To be sure, in \cite{ArtemiouLi2009}, the authors used the variance of the principal components to prove that it is more probable that the response $Z$ is closer to the principal component $Y_i$ than to the principal component $Y_j$, for $i<j$.
 \item Second, we are not projecting the space to any lower dimension when using principal components. As it is known, reducing the space in contexts of regression might be problematic (see for instance \cite{Joliffe1982} and \cite{HadiLing1998}). But Theorem \ref{TeoLines} is true for all the $p$ PC in the space.
 \item Remarks \ref{R8} and \ref{R9} also admit natural extensions to this scenario of random variables and principal components. So we get, for instance, that the probability of \emph{all} principal components explaining better the response is exponentially higher than the probability of \emph{all} the original input explaining it better. Of course, not even principal components are immune to the curse of dimensionality, but it shows that it affects the PC explanatory power at a lower speed.
\end{itemize}

\section{Discussion}
Principal components analysis (PCA) is a widely used technique but some mystery has remained as to why it's a reasonable thing to do when modeling a response-predictor relationship since it was introduced in \cite{Hotelling1957} and \cite{Kendall1957} about 60 years ago.  We offer a partial answer to this question. A more comprehensive answer will have to look at the probability of the response being closer to at least one of the PC than to all of the original input variables, but this is still an open problem.
\vskip 0.5in
\noindent
{\bf Acknowledgements}:  All authors supported in part by NIH grant NCI R01-CA160593A1. We would like to thank Rob Tibshirani, Steve Marron and Hemant Ishwaran for helpful discussions of the work. DD would like to thank Juan Saenz for his suggestions on how to prove the results and Federico Ardila for an early very useful comment on Proposition \ref{CondExp}.


\begin{thebibliography}{9}

\bibitem{ArtemiouLi2009}
\textsc{Artemiou A., Li B.} (2009)
On Principal Components and Regression: A Statistical Explanation of a Natural Phenomenom.
\textit{Statistica Sinica}
\textbf{19} 1557--1565.

\bibitem{Cook2007}
\textsc{Cook R. D.} (2007) 
Fisher Lecture: Dimension reduction in regression.
\textit{Statist. Sci.}
\textbf{22} 1--40.

\bibitem{Cox1968}
\textsc{Cox D. R.} (1968)
Notes on some aspects of regression analysis.
\textit{J. Roy. Ststist. Soc. Ser. A}.
\textbf{131} 265--279.

\bibitem{Joliffe1982}
\textsc{Joliffe I.} (1982)
A Note on the Use of Principal Components in Regression.
\textit{Journal of the Royal Statistical Society. Series C (Applied Statistics)}
\textbf{31} (3), 300--303.

\bibitem{HadiLing1998}
\textsc{Hadi S., Ling R.F.} (1998)
Some Cautionary Notes on the Use of Principal Components Regression.
\textit{The American Statistician}
\textbf{52} (1), 15--19.

\bibitem{Hocking1976}
\textsc{Hocking R. R.} (1976)
The analysis and selection of variables in linear regression.
\textit{Biometrics}
\textbf{32} 1--49.

\bibitem{Hotelling1957}
\textsc{Hotelling H.} (1957)
The relationship of the newer multivariate statistical methods to factor analysis.
\textit{Brit. J. Stat. Psychol.}
\textbf{10} 69--70.

\bibitem{Kendall1957}
\textsc{Kendall M. G.} (1957)
\textit{A Course in Multivariate Analysis.} Griffin.

\bibitem{MostellerTuckey1977}
\textsc{Mosteller F., Tuckey J. W.} (1977)
\textit{Data Analysis and Regression.} Addison-Wesley.

\bibitem{Scott1992}
\textsc{Scott D.} (1992)
\textit{Multivariate density estimation.} Wiley.

\end{thebibliography}
\end{document}